\documentclass[11pt]{article}

\usepackage{graphicx}
\usepackage{flafter}
\usepackage{geometry}
\usepackage{amsmath}
\usepackage{amsthm}
\usepackage{mathrsfs}
\usepackage{amsfonts}
\usepackage{amssymb}
\usepackage{enumerate}
\usepackage{tikz}
\usepackage{ytableau}
\allowdisplaybreaks
\usepackage{verbatim}
\usepackage[colorlinks, linkcolor=blue, anchorcolor=blue, citecolor=blue, urlcolor=blue]{hyperref}
\usepackage{color}
\usepackage{graphicx}
\usepackage{tikz}
\usepackage{soul, ulem, cancel}

\makeatletter \@addtoreset{equation}{section}

\newcommand{\qbinomial}[2]{\genfrac{[}{]}{0pt}{}{#1}{#2}_q}

\newtheorem{thm}{Theorem}[section]

\newtheorem{coro}{Corollary}[section]

\newtheorem{conj}{Conjecture}[section]

\begin{document}

\title{Analytic properties arising from the Baxter numbers}

\author{
	Hanqian Fang\footnote{KLMM, Academy of Mathematics and Systems Science, Chinese Academy of Sciences, Beijing 100190, P.R. China.  Email: fanghanqian22@mails.ucas.ac.cn.},\
	Candice X.T. Zhang\footnote{KLMM, Academy of Mathematics and Systems Science, Chinese Academy of Sciences, Beijing 100190, P.R. China. Email: zhangxutong@amss.ac.cn.},\
	James J.Y. Zhao\footnote{School of Accounting, Guangzhou College of Technology and Business,
		Foshan 528138, P.R. China.  Email: zhao@gzgs.edu.cn.}}
\maketitle

\begin{abstract}
	Baxter numbers are known as the enumeration of Baxter permutations and numerous other discrete structures, playing a significant role across combinatorics, algebra, and analysis.
	In this paper, we focus on the analytic properties related to Baxter numbers.
	We prove that the descent polynomials of Baxter permutations have interlacing zeros, which is a property stronger than real-rootedness.
	Our approach is based on Dilks' framework of $(q,t)$-Hoggatt sums, which is a $q$-analog for Baxter permutations.
	Within this framework, we show that the family of $(1,t)$-Hoggatt sums satisfies the interlacing property using fundamental results on Hadamard products of polynomials.
	For Baxter numbers, we prove their asymptotic $r$-log-convexity via asymptotic expansions of $P$-recursive sequences. In particular, we confirm their $2$-log-convexity using symbolic computation techniques.
	
	\noindent {\bf AMS Classification 2020:} {26C10; 11B83; 05A10}
	
\end{abstract}

\section{Introduction}

The Baxter permutations were first introduced by Glen Baxter~\cite{Baxter1964} in 1964 in the context of studying the fixed points of the composition of commuting functions. A Baxter permutation on $[n]$ is a permutation $\pi=\pi_1\pi_2\cdots\pi_n\in \mathfrak{S}_n$ which avoids all subsequences $\pi_{i}\pi_j\pi_{j+1}\pi_k$ (for indices $i<j<k-1$) satisfying either of the following patterns:
\[\pi_{j+1}<\pi_i<\pi_k<\pi_j\quad\text{ or }\quad\pi_{j}<\pi_k<\pi_i<\pi_{j+1}.\]
These permutations have significance in combinatorics, algebra and analysis. Various bijections have been built between Baxter permutations and other combinatorial objects, such as twisted Baxter permutations~\cite{Law-Reading2012}, 
pairs of twin binary trees on $n$ nodes~\cite{Giraudo2012}, plane partitions in a box~\cite{Dilks2012}, and so on.
These objects form the well-known ``Baxter family". In addition, it turns out that these objects possess elegant algebraic properties and thus become suitable candidates to index the bases of some Hopf algebras.
For more details, see~\cite{Reading2005,Law-Reading2012,Chatel-Pilaud2017}.
From an analytic perspective, inspired by the bijection between Baxter permutations and plane partitions in a box, Dilks~\cite{Dilks2015} gave a natural $q$-analog for Baxter permutations, called the $(q,t)$-Hoggatt sums, and further proved their $q$-gamma nonnegativity.
From a matrix perspective, Mao and Shi~\cite{Mao-Shi2025} studied various properties on log-behaviors in the $m$-Hoggatt triangle, whose row generating functions are the $(1,t)$-Hoggatt sums. Furthermore, they also obtained the asymptotic normality of the row sequences and total positivity of the $m$-Hoggatt triangle.
In this paper, we delve deeper into the analytic properties arising from the Baxter permutations.

Baxter permutations are counted by {Baxter numbers} $B_n$~\cite[\textsc{A001181}]{oeis} with the formula
\begin{equation}\label{eq:B_n}
	B_n = \sum_{k=1}^n\frac{2}{n(n+1)^2}\binom{n+1}{k-1} \binom{n+1}{k} \binom{n+1}{k+1}.
\end{equation}
This formula was first established by Chung {\it et al.}~\cite{CGHK1978} with analytical methods, and later by Viennot~\cite{Viennot1981} through combinatorial bijections. The sum terms in Equation~\eqref{eq:B_n}, defined as 
\[D_{n,\,k}=\frac{2}{n(n+1)^2}\binom{n+1}{k-1} \binom{n+1}{k} \binom{n+1}{k+1},\]
also have corresponding combinatorial interpretations within the Baxter family. Mallows~\cite{Mallows1979} first showed that $D_{n,\,k}$ counts the reduced Baxter permutations on $[n]$ with exactly $k$ rises. Felsner \textit{et al.} further presented a unifying presentation for the bijections between the objects in Baxter family, which implies that $D_{n,\,k}$ also counts the Baxter permutations on $[n]$ with $k-1$ descents and $n-k$ rises~\cite[Proposition 6.8]{FFNO2011}, 
twin pairs of binary trees with $k$ left leaves and $n-k+1$ right leaves~\cite[Theorem 5.6]{FFNO2011},
and so on. These results demonstrate that the study of the numbers $D_{n,\,k}$ is not only intriguing in its own right, but also crucial for understanding Baxter numbers $B_n$. Therefore, we define the {\it Baxter polynomial} $PB_n(t)$ to be the generating function of $D_{n,\,k}$, i.e.,
\begin{equation*}
	PB_n(t):
	= \sum_{k=1}^n D_{n,\,k} t^{k-1}
	= \sum_{k=1}^n \frac{2}{n(n+1)^2}\binom{n+1}{k-1} \binom{n+1}{k}
	\binom{n+1}{k+1} t^{k-1}
\end{equation*}
for $n\geq 1$. Note that $PB_n(t)$ can be equivalently interpreted as the descent polynomials of Baxter permutations.

As stated above, Dilks~\cite{Dilks2015} placed Baxter polynomials within a broader framework of the $(q,t)$-Hoggatt sums, defined by
\begin{align*}
	H_n^{[m]}(q,t)=\sum_{\substack{k+l=n-1}}M(k,l,m;q)t^k,
\end{align*}
where
\begin{align*}
	M(k,l,m;q)=q^{m\binom{k+1}{2}}\sum_{\pi\in\textrm{PP}_{k,\,l,\,m}}q^{|\pi|},
\end{align*}
and $\textrm{PP}_{k,\,l,\,m}$ denotes the set of all plane partitions in a $k\times l \times m$ box.
These polynomials are the $q$-analogs of the Hoggatt sums, originally introduced by Fielder and Alford~\cite{Fielder-Alford-1989} in the study of Hoggatt's conjecture related to Pascal's triangle and Baxter permutations.
For the special case $q=1$, the $(1,t)$-Hoggatt sum reduces to
\begin{align*}
	H_n^{[m]}(1,t)=\sum_{k=0}^{n-1}\frac{\prod_{i=0}^{m-1}\binom{n+m-2}{k+i}}{\prod_{j=1}^{m-1}\binom{n+m-2}{j}}t^k,
\end{align*}
where the coefficient counts the number of plane partitions fitting in a $k\times (n-1-k) \times m$ box, see~\cite[Theorem 7.21.7]{Stanley2001}.
Notably, when $m=3$, the polynomial $H_n^{[m]}(1,t)$ specializes to the Baxter polynomial $PB_n(t)$. In the remainder of this paper, we call the $(1,t)$-Hoggatt sums $H_n^{[m]}(1,t)$ the {\it Hoggatt polynomials}.

For all $m\ge 1$, the real-rootedness of the Hoggatt polynomials $H_n^{[m]}(1,t)$ has been obtained by Dilks~\cite[Theorem 3.8]{Dilks2015} via the theory of multiplier sequences~\cite{Branden-2006}.
Motivated by this, one of the main objectives of this paper is to explore the interlacing property of the polynomial sequence $\{H_n^{[m]}(1,t)\}_{n\geq1}$ for any fixed $m\ge 1$. 

For the purpose of describing our main theorem, we need the following definitions and notations. Given two real-rooted polynomials $f(t)$, $g(t)$ with real coefficients, let $\{u_i\}$ and $\{v_j\}$ be all zeros of $f(t)$ and $g(t)$ in weakly decreasing order, respectively.
We say that $g(t)$ {interlaces} $f(t)$ if $\deg f(t)=\deg g(t)+1=n$ and
\begin{align}\label{eq:def-interlac-2}
	{u_{n}\leq v_{n-1}\leq u_{n-1}\leq \cdots
		\leq v_{2}\leq u_{2}\leq v_{1}\leq u_{1}}.
\end{align}
We say that $g(t)$ {alternates left of} $f(t)$ if $\deg f(t)=\deg g(t)=n$ and
\begin{align}\label{eq:def-interlac-1}
	{v_n\leq u_n\leq v_{n-1}\leq u_{n-1}\leq \cdots
		\leq v_2\leq u_2\leq v_1\leq u_1}.
\end{align}
Denote by $g(t)\preceq f(t)$ if either $g(t)$ interlaces $f(t)$ or $g(t)$ alternates left of $f(t)$.
We say that $g(t)$ {strictly interlaces} (resp., {strictly alternates left of}) $f(t)$ if no equality holds in~\eqref{eq:def-interlac-2} (resp.,~\eqref{eq:def-interlac-1}).
Write $g(t)\prec f(t)$ if $g(t)$ strictly interlaces or strictly alternates left of $f(t)$.
For convenience, we also let $a\preceq bx+c$ for any nonnegative real numbers $a,b,c$, and $f(t)\preceq 0$, $0\preceq f(t)$ for any real polynomial $f(t)$ with only real zeros.
A sequence of real-rooted polynomials $\{f_n(t)\}_{n\geq 0}$ is said to be a {Sturm sequence} if
$$
{f_0(t)\prec f_1(t)\prec \cdots \prec  f_n(t)\prec \cdots},
$$
and it is said to be a {generalized Sturm sequence} if
$$
{f_0(t)\preceq f_1(t)\preceq \cdots \preceq  f_n(t)\preceq \cdots}.
$$
We can now give our main results as follows.
\begin{thm}\label{thm:interlace-PBn}
	The sequence of Baxter polynomials $\{PB_n(t)\}_{n\geq 1}$ is a Sturm sequence.
\end{thm}

In fact, we obtain a more general result in the framework of Hoggatt polynomial sequence $\{H_n^{[m]}(1,t)\}_{n\ge 1}$ by induction on $m$.

\begin{thm}\label{thm:interlace-HP}
	For $m=1$, the sequence of Hoggatt polynomials $\{H_n^{[1]}(1,t)\}_{n\geq 1}$ is a generalized Sturm sequence.
	For any integer $m\geq 2$, the sequence $\{H_n^{[m]}(1,t)\}_{n\geq 1}$ is a Sturm sequence.
\end{thm}

Since a (generalized) Sturm sequence must be a real-rooted polynomial sequence, we also provide an alternative proof for the real-rootedness result established by Dilks~\cite[Theorem 3.8]{Dilks2015}.
Note that the coefficients of a real-rooted polynomial sequence are usually asymptotically normal. For instance, the asymptotic normality of the numbers $D_{n,\,k}$ was obtained by Zhao~\cite{Zhao}. Recently, Mao and Shi~\cite{Mao-Shi2025} extended this result to the coefficients $M(k,n-k-1,m;1)$ of Hoggatt polynomials.
In addition, the real-rootedness of a polynomial also implies the log-concavity of its coefficient sequence via the well-known Newton inequality.
Furthermore, if this polynomial has only negative zeros, Br\"{a}nd\'{e}n~\cite{Branden2011} showed that its coefficient sequence is infinitely log-concave, a stronger definition introduced by Boros and Moll~\cite{Boros-Moll-2004}. Utilizing this result, Mao and Shi~\cite[Theorem 2.3]{Mao-Shi2025} remarked that the sequence $\{M(k,n-k-1,m;1)\}_{k=0}^{n-1}$ is infinitely log-concave for any integers $n\geq 1$ and $m\geq 1$. As a special case, this implies the infinite log-concavity of the sequence $\{D_{n,\,k}\}_{k=1}^n$.

Our second main result focuses on the log-convexity property, a dual notion of log-concavity, of the Baxter numbers $B_n$.
Recall that an infinite sequence $\{a_n\}_{n \ge 0}$ is said to be {log-convex} if $a_n^2\leq a_{n-1}a_{n+1}$ for $n\ge 1$.
Let $\mathscr{L}$ be the operator defined by $$\mathscr{L}\left(\{ a_n\}_{n\ge 0}\right)=\{a_na_{n+2}-a_{n+1}^2 \}_{n\ge 0}.$$
A sequence $\{a_n\}_{n\ge 0}$ is said to be {$r$-log-convex} if the $j$-th iterative sequence $\mathscr{L}^j\left(\{a_n\}_{n\ge 0}\right)$ is nonnegative for all $1\le j \le r$.
Furthermore, $\{a_n\}_{n\ge 0}$ is said to be {infinitely log-convex} if it is $r$-log-convex for all $r\ge 1$.
Hou and Zhang~\cite{Hou-Zhang2019} introduced the concept of {asymptotic $r$-log-convexity} of $\{a_n\}_{n\geq 0}$ which requires that for each $1\leq j\leq r$, the sequence $\mathscr{L}^j\left(\{a_n\}_{n\geq 0}\right)$ is nonnegative for all but finitely many initial terms. The log-convexity of the Baxter number sequence $\{B_n\}_{n\geq 1}$ was obtained in~\cite[Theorem 4.18]{Doslic-Veljan2008} by Do\v{s}li\'{c} and Veljan. In this paper, we focus on the (asymptotic) $r$-log-convexity of $\{B_n\}_{n\geq1}$.

\begin{thm}\label{thm:asym-r-lgcvBn}
	The sequence $\{B_n\}_{n\geq 1}$ is asymptotically $r$-log-convex for any integer $r\geq 1$. In particular, for $r=2$, the sequence $\{B_n\}_{n\geq 1}$ is $2$-log-convex.
\end{thm}

The remainder of this paper is organized as follows.
In Section~\ref{S-Hoggatt}, we prove Theorem~\ref{thm:interlace-HP} by applying Garloff and Wagner's results on Hadamard products of interlacing polynomials and Dimitrov's work concerning derivatives of interlacing polynomials.
Section~\ref{S-rlgcvx} is devoted to the study of (asymptotic) $r$-log-convexity of Baxter numbers. We show a proof of Theorem~\ref{thm:asym-r-lgcvBn} with some tools in asymptotic expansions of $P$-recursive sequences given by Hou and Zhang. In Section~\ref{sec-conclusion}, we conclude this paper by proposing three conjectures related to $(q,t)$-Hoggatt sums and Baxter numbers.

\section{Interlacing property of Hoggatt polynomials}\label{S-Hoggatt}

The aim of this section is to complete the proof of Theorem~\ref{thm:interlace-HP}. To this end, we first recall two operations, Hadamard product and polynomial derivative, which preserve the interlacing property.

The Hadamard product of two polynomials $f(t)=\sum_{i=0}^{n}a_{i}t^{i}$ and $g(t)=\sum_{i=0}^{m}b_{i}t^{i}$ is defined to be
\begin{align}\label{defi:Hadamard-prd}
	f(t)*g(t)=\sum_{i=0}^{k}a_{i}b_{i}t^{i},
\end{align}
where $k=\min\{m,n\}$.
Mal\'{o}~\cite{Malo-1895} first proved that the Hadamard product preserves the real-rootedness of polynomials. Later, Garloff and Wagner~\cite{Garloff-Wagner-1996} extended Mal\'{o}'s result as follows.

\begin{thm}[{\cite[Theorem 4]{Garloff-Wagner-1996}}]\label{thm-GW}
	Let $f(t),\,g(t),\,p(t),\,q(t) \in \mathbb{R}[t]$ with positive leading coefficients and only nonpositive zeros.
	\begin{itemize}
		\item[(1)] If $f(t)*p(t)\neq 0$, then $f(t)*p(t)$ has only nonpositive zeros which are simple except possibly at the origin.
		\item[(2)] If $g(t)\preceq f(t)$ and $q(t)\preceq p(t)$, then we have $$g(t)*q(t)\preceq f(t)*p(t).$$
		\item[(3)] If $g(t)\preceq f(t)$ and $q(t)\preceq p(t)$, $f(t),\,g(t),\,p(t)$ and $q(t)$ are all nonzero, and neither $f(t)=c_1g(t)$ and $p(t)=c_2q(t)$ nor $f(t)=c_1tg(t)$ and $p(t)=c_2tq(t)$ for any $c_1,\,c_2>0$, then we have
		\begin{align}\label{eq-gcd-r}
			\gcd(g(t)*q(t),f(t)*p(t))=t^r,
		\end{align}
		where $r$ is the maximum multiplicity of $t=0$ as a zero of the polynomials $g(t)$ and $q(t)$.
	\end{itemize}
\end{thm}

To better demonstrate Theorem~\ref{thm:interlace-HP}, we simplify the above theorem as the following corollary.

\begin{coro}\label{coro-GS-S}
	Let $f(t),\,g(t),\,p(t),\,q(t) \in \mathbb{R}[t]$ with positive leading coefficients and only negative zeros.
	If $g(t)$ interlaces $f(t)$ and $q(t)$ interlaces $p(t)$, then we have $$g(t)*q(t)\prec f(t)*p(t).$$
\end{coro}
\begin{proof}
	According to~(2)~of~Theorem~\ref{thm-GW}, one can first deduce that
	\begin{align}\label{eq-fgpq}
		g(t)*q(t)\preceq f(t)*p(t).
	\end{align}
	Furthermore, under the given conditions, we have $\deg f(t)=\deg g(t) +1$, $\deg p(t)=\deg q(t)+1$, and $f(t),\,g(t),\,p(t),\,q(t)$ do not vanish at the origin. These imply that $f(t)\neq c_1g(t)$, $p(t)\neq c_2q(t)$, $f(t)\neq c_3 t g(t)$ and $p(t)\neq c_4 t q(t)$ for any $c_1,\,c_2,\,c_3,\,c_4>0$. Utilizing~(3)~of~Theorem~\ref{thm-GW}, we can obtain that $r=0$ in Equation~\eqref{eq-gcd-r} and thus
	$$\gcd(g(t)*q(t),f(t)*p(t))=t^r=1.$$
	As a result, the relation~\eqref{eq-fgpq} holds strictly.
\end{proof}

Another crucial tool is a celebrated result on polynomial derivative given by Markov~\cite{Markov-1892}, which was restated by Dimitrov~\cite[Page 70]{Dimitrov-2012}.
\begin{thm}[{\cite[\S 16]{Markov-1892}}]\label{thm:Markov}
	Let $f(t)$ and $g(t)$ be real algebraic polynomials with only real and distinct zeros.
	If $g(t)\prec f(t)$, then $g'(t) \prec f'(t)$.
\end{thm}

Now we are ready to prove Theorem~\ref{thm:interlace-HP}, which includes Theorem~\ref{thm:interlace-PBn} as a special case.
\begin{proof}[Proof of Theorem~\ref{thm:interlace-HP}]	
For convenience, we set
\begin{equation}\label{defi:Fnell}
	F_{n,\,m}(t):=\sum_{k=0}^{n-m+1}\binom{n}{k}\binom{n}{k+1}\cdots \binom{n}{k+m-1}t^k.
\end{equation}
This simplifies the expression of $H_n^{[m]}(1,t)$ since
\begin{align*}
	H_{n}^{[m]}(1,t)
	=&\,\frac{1}{\binom{n+m-2}{1}\binom{n+m-2}{2}\cdots\binom{n+m-2}{m-1}}
	\sum_{k=0}^{n-1}\binom{n+m-2}{k}\binom{n+m-2}{k+1}\cdots\binom{n+m-2}{k+m-1} t^k\\
	=&\,\frac{1}{\binom{n+m-2}{1}\binom{n+m-2}{2}\cdots\binom{n+m-2}{m-1}}F_{n+m-2,\,m}(t).
\end{align*}
	Then the assertion for $m=1$ is trivial since $H_n^{[1]}(1,t)=F_{n-1,\,1}(t)=(t+1)^{n-1}$ for $n\geq 1$.
	Now it suffices to prove that for any fixed $m\ge 2$, the sequence $\{F_{n,\,m}(t)\}_{n\geq m-1}$ is a Sturm sequence by induction on $m\ge 2$.
	
	For $m=2$, we have $(n+1)F_{n,\,2}(t)=(F_{n,\,1}(t)* F_{n+1,\,1}(t))'$
	by the definitions~\eqref{defi:Hadamard-prd} and~\eqref{defi:Fnell}.
	It is evident that $F_{n-1,\,1}(t)$ interlaces $F_{n,\,1}(t)$ for any $n\geq2$.
	Thus the sequence $\{F_{n,\,2}(t)\}_{n\geq 1}$ is a Sturm sequence by Corollary~\ref{coro-GS-S} and Theorem~\ref{thm:Markov}.
	
	For $m\ge 2$, suppose that $\{F_{n,\,m-1}(t)\}_{n\geq m-2}$ is a Sturm sequence.
	We proceed to prove that for any $n\ge m$,
	\begin{align}\label{intlc-Fn-1Fn}
		F_{n-1,\,m}(t)\prec F_{n,\,m}(t).
	\end{align}
	To achieve this, we construct an auxiliary polynomial $G_{n,\,m}(t)$ as follows:
	\begin{align*}
		G_{n,\,m}(t)&=F_{n,\,1}(t)* F_{n-1,\,m-1}(t)\\
		&=\sum_{k=0}^{n-m+1}\binom{n}{k}\cdot\left(\binom{n-1}{k}\binom{n-1}{k+1}\cdots \binom{n-1}{k+m-2}\right)\cdot t^k.
	\end{align*}
	Since the polynomials $F_{n,\,1}(t)$ and $F_{n-1,\,m-1}(t)$ have nonzero constant and their coefficients are all positive, it is clear that they have only negative zeros.
	Thus the polynomial $G_{n,\,m}(t)$ has only nonpositive zeros by~(1)~of Theorem~\ref{thm-GW}. 
	Besides, we have $\deg F_{n,\,m}(t)= n-m+1$ for any $n\ge m-1$.
	Combining the fact that $F_{n,\,1}(t) \preceq F_{n+1,\,1}(t)$ and the inductive hypothesis
	\begin{align*}
		F_{n-1,\,m-1}(t)\prec F_{n,\,m-1}(t),
	\end{align*}
	we apply Corollary~\ref{coro-GS-S} to conclude that
	\begin{align*}
		G_{n,\,m}(t)\prec G_{n+1,\,m}(t).
	\end{align*}
	This implies that
	\begin{align}\label{intlc-Gn'x}
		G_{n,\,m}'(t)\prec G_{n+1,\,m}'(t)
	\end{align}
	by Theorem~\ref{thm:Markov}.
	We see that
	\begin{align}
		G_{n,\,m}'(t)&=\sum_{k=0}^{n-m+1}k\binom{n}{k}\binom{n-1}{k}\binom{n-1}{k+1}\cdots \binom{n-1}{k+m-2}t^{k-1}\nonumber\\[5pt]
		&=\sum_{k=0}^{n-m}(k+1)\binom{n}{k+1}\binom{n-1}{k+1}\binom{n-1}{k+2}\cdots \binom{n-1}{k+m-1 }t^{k}\nonumber\\[5pt]
		&=n\sum_{k=0}^{n-m}\binom{n-1}{k}\binom{n-1}{k+1}\binom{n-1}{k+2}\cdots \binom{n-1}{k+m-1}t^{k}\nonumber\\[5pt]
		&=n F_{n-1,\,m}(t).\label{eq:Gn'-Fn1x}
	\end{align}
	From the relations~\eqref{intlc-Gn'x} and~\eqref{eq:Gn'-Fn1x}, we obtain~\eqref{intlc-Fn-1Fn}. It follows that the polynomial sequence $\{F_{n,\,m}(t)\}_{n\ge m-1}$ is a Sturm sequence and so is $\{H_n^{[m]}(1,t)\}_{n\ge 1}$ for $m\ge 2$.
\end{proof}

\section{Asymptotic $r$-log-convexity of the Baxter numbers}\label{S-rlgcvx}

In this section, we complete the proof of Theorem~\ref{thm:asym-r-lgcvBn} by Hou and Zhang’s method on asymptotic analysis. Let us begin with a brief overview of their technique~\cite{Hou-Zhang2019}.

The asymptotic expansions of $P$-recursive sequences were given by Birkhoff and Trjitzinsky~\cite{Birkhoff-Trjitzinsky1933} and later developed by Wimp and Zeilberger~\cite{Wimp-Zeilberger1985}.
Based on their results, Hou and Zhang~\cite{Hou-Zhang2019} further established a symbolic method to determine whether a $P$-recursive sequence is (asymptotically) $r$-log-convex.

Let $\{a_n\}_{n\geq 0}$ be a sequence of real numbers. Suppose that there exist real coefficients $\ell_0,\,\ell_1,\ldots,\ell_m$ and real exponents $\alpha_0<\alpha_1<\cdots<\alpha_m $
such that
\begin{align*}
	\lim_{n\rightarrow \infty} n^{\alpha_m}\left(a_n-\sum_{i=0}^m \frac{\ell_i}{n^{\alpha_i}}\right)=0.
\end{align*}
Then a {Puiseux-type approximation} of $a_n$, denoted by $a_n \approx b_n$, is defined by the summation
\begin{align*}
	b_n=\sum_{i=0}^m \frac{\ell_i}{n^{\alpha_i}}.
\end{align*}
Thus, the asymptotic expansion of $a_n$ can be written in the standard little-o notation as
\begin{align*}
	a_n=\frac{\ell_0}{n^{\alpha_0}}+\frac{\ell_1}{n^{\alpha_1}}+\cdots
	+o\left(\frac{1}{n^{\alpha_m}}\right).
\end{align*}
Define an operator $\mathscr{R}$ on $\{a_n\}_{n\ge 0}$ as $\mathscr{R}(a_n)=a_{n+1}/a_n$.
Clearly, $\mathscr{R}^2 (a_n)=a_{n}a_{n+2}/a_{n+1}^2$.
Hou and Zhang~\cite{Hou-Zhang2019} derived the following criterion for asymptotic $r$-log-convexity.

\begin{thm}[{\cite[Theorem 2.2]{Hou-Zhang2019}}]\label{thm:Hou-Zhang-1}
	Let $\{a_n\}_{n\geq 0}$ be a positive sequence such that the ratio $\mathscr{R}^2 (a_n)$ has a Puiseux-type approximation of the form
	\begin{align}\label{eq:P-ty-1}
		\mathscr{R}^2 (a_n) = 1 + \frac{c}{n^{\alpha}} +\cdots + o \left(\frac{1}{n^\beta}\right),
	\end{align}
	where $0< \alpha \leq \beta$.
	If $c>0$ and $\alpha\geq 2$, then $\{a_n\}_{n\geq 0}$ is asymptotically $\left(\lfloor \frac{\beta-\alpha}{2}\rfloor+1\right)$-log-convex.	
\end{thm}
This reduces the proof of the asymptotic result in Theorem~\ref{thm:asym-r-lgcvBn}
to finding a Puiseux-type approximation of $\mathscr{R}^2 (B_n)$ that satisfies~\eqref{eq:P-ty-1} with $c>0$ and $\alpha\geq2$ for any given integer $\beta\geq 2$.
In fact, there exists a systematic theory for handling problems of this type associated with $P$-recursive sequences. 
Recall that a sequence $\{a_n\}_{n\ge 0}$ is said to be {$P$-recursive} of order $j$ if it satisfies a recurrence relation of the form
\begin{align}\label{rec:P-r}
	a_{n+j}=R_0(n) a_n + R_1(n) a_{n+1} + \cdots + R_{j-1}(n) a_{n+j-1}
\end{align}
where each $R_i(n)$ is a rational function of $n$, see Stanley~\cite[Section 6.4]{Stanley2001}.
Suppose that the sequence $\{a_n\}_{n\geq 0}$ is $P$-recursive, then $a_n$ can be asymptotically expressed as a finite linear combination of terms of the form
\begin{align*}
	e^{Q(\rho,n)}\, s(\rho,n),
\end{align*}
where
\begin{align*}
	Q(\rho,n) & = \mu_0 n \log n + \sum_{j=1}^{\rho} \mu_j n^{j/\rho}, \\
	s(\rho,n) & = n^{\nu} \sum_{j=0}^{\delta-1}(\log n)^j \sum_{k=0}^{K} \lambda_{k,\,j} n^{-k/\rho},
\end{align*}
with $\rho,\, \delta,\,K$ being positive integers and $\mu_j,\, \nu,\, \lambda_{k,\,j}$ being complex numbers.
See Birkhoff and Trjitzinsky~\cite{Birkhoff-Trjitzinsky1933}, Wimp and Zeilberger~\cite{Wimp-Zeilberger1985} for more details.

Furthermore, Hou and Zhang~\cite{Hou-Zhang2019} also gave a symbolic method for finding an explicit $N_r$ such that $\{a_n\}_{n\geq N_r}$ is $r$-log-convex, when it is asymptotically log-convex.
They implemented their algorithms in a {\tt Mathematica} package {\tt P-rec.m}, and the related command is
\begin{align}\label{com:rLB}
	{\tt rLogBound[L,\ n,\ N,\ ini\_val,\ r,\ \eta]}
\end{align}
where
$$
L=\left(R_0(n) + R_1(n) {\tt N} + \cdots + R_{j-1}(n) {\tt N}^{j-1}\right)-{\tt N}^j
$$
corresponds to the recurrence~\eqref{rec:P-r}, ${\tt N}$ is the shift operator such that ${\tt N}(a_n)=a_{n+1}$, {\tt ini\_val} is a list of initial values of $\{a_n\}_{n\geq 0}$, and $\eta$ is the number of terms in a Puiseux-type approximation expansion of $\mathscr{R}(a_n)$. As mentioned by Hou and Zhang~\cite{Hou-Zhang2019}, when $\eta$ is large enough, we shall get a symbolic proof of the $r$-log-convexity of a given sequence.

We are now in a position to prove Theorem~\ref{thm:asym-r-lgcvBn}, following the approaches applied in~\cite[Corollary 3.2 and Theorem 4.2]{Hou-Zhang2019}.

\begin{proof}[Proof of Theorem~\ref{thm:asym-r-lgcvBn}]
	
	To begin with, we observe that the sequence $\{B_n\}_{n\geq1}$ is $P$-recursive, satisfying the recurrence relation
	\begin{align}\label{eq:rec-Bn}
		(n+4)(n+5)B_{n+2}=(7n^2+35n+40)B_{n+1}+8n(n+1)B_{n},
	\end{align}
	as recorded by Richard L. Ollerton in the OEIS~\cite[A001181]{oeis}. Since $B_n$ can be written as a sum of hypergeometric terms $D_{n,\,k}$ over $k$, it is easy to verify~\eqref{eq:rec-Bn} using Zeilberger's algorithm~\cite{Zeilberger1991}.
		
	Then we can apply Theorem~\ref{thm:Hou-Zhang-1} to prove the asymptotic $r$-log-convexity of $\{B_n\}_{n\geq1}$.
	Clearly, $B_n>0$ for all $n\geq 1$.
	We find a Puiseux-type approximation of $B_n$ from the recurrence relation~\eqref{eq:rec-Bn} by using {\tt P-rec.m}\footnote{See \href{http://faculty.tju.edu.cn/146004/zh\_CN/lwcg/29931/content/37248.htm\#lwcgl}
		{https://faculty.tju.edu.cn/146004/zh\_CN/lwcg/29931/content/37248.htm\#lwcgl}.}
	given by Hou~\cite{Hou-Zhang2019}. For more alternative packages, see {\tt AsyRec}\footnote{See \href{https://sites.math.rutgers.edu/~zeilberg/mamarim/mamarimhtml/asy.html}
		{https://sites.math.rutgers.edu/\url{~}zeilberg/mamarim/mamarimhtml/asy.html}.}
	given by Zeilberger~\cite{Zeilberger2016}, 
	and {\tt Asymptotics.m}\footnote{See
		\href{http://www.kauers.de/software.html}{http://www.kauers.de/software.html}.} given by Kauers~\cite{Kauers2011}.
	First let
	\begin{align*}
		{\tt L=(n+4)(n+5)N^2-(7n^2+35n+40)N-8n(n + 1)}.
	\end{align*}
	By running the command
	\begin{align}\label{eq:command}
		{\tt Asy[L,\  n,\  N,\  K]}
	\end{align}
	for $K=2$,
	we find out an asymptotic expansion of $B_n$:
	\begin{align}\label{eq-Pt-Bn}
		B_n = C \cdot 8^n n^{-4} \left(1 - \frac{22}{3n}+ \frac{955}{27 n^2} +o\left(\frac{1}{n^2}\right)\right),
	\end{align}
	where $C$ is a constant. In fact, $C={2^5}/{(\sqrt{3}\pi)}$, see~\cite[Eq. 20]{CGHK1978}.
	Therefore, $\mathscr{R}^2 (B_n)$ has the form
	\begin{align}\label{eq:2.7}
		\mathscr{R}^2 (B_n)=\frac{B_{n}B_{n+2}}{B_{n+1}^2}
		= 1+\frac{4}{n^2} + o\left(\frac{1}{n^2}\right).
	\end{align}
	Comparing~\eqref{eq:2.7} with~\eqref{eq:P-ty-1}, it is clear that $c=4>0$, $\alpha=2>0$, $\beta=2\geq \alpha$, and hence $\mathscr{R}^2 (B_n)$ has a Puiseux-type approximation of the form~\eqref{eq:P-ty-1}. This implies that $\{B_n\}_{n\geq1}$ is asymptotically $1$-log-convex by Theorem~\ref{thm:Hou-Zhang-1}. For any $r\geq2$, taking $K$ to be $2r$ in the command~\eqref{eq:command}, one can obtain a Puiseux-type approximation of $B_n$ in the form of~\eqref{eq:P-ty-1} with the same $\alpha$ and $c$ as those in~\eqref{eq-Pt-Bn}. This leads to an asymptotic expansion of $\mathscr{R}^2 (B_n)$ as
	\begin{align*}
		\mathscr{R}^2 (B_n)=\frac{B_{n}B_{n+2}}{B_{n+1}^2}
		= 1+\frac{4}{n^2} +\cdots+ o\left(\frac{1}{n^{2r}}\right).
	\end{align*}
	As a result, we finally obtain the asymptotic $r$-log-convexity of $\{B_n\}_{n\geq1}$ by Theorem~\ref{thm:Hou-Zhang-1} for any positive integer $r$.
	
	For the $2$-log-convexity of $\{B_n\}_{n\ge 1}$, the symbolic proof is almost one-line by running the command given in~\eqref{com:rLB}.
	The command
	\begin{align*}
		{\tt rLogBound[L,\ n,\ N,\ \{1,\ 2,\ 6,\ 22,\ 92\},\ 2,\ 5]}
	\end{align*}
	outputs {\tt \{True, 13069\}}, which means that $\{B_n\}_{n\geq 13069}$ is $2$-log-convex.
	Since the log-convexity of $\{B_n\}_{n\geq 1}$ has been proved by Do\v{s}li\'{c} and Veljan~\cite[Theorem 4.18]{Doslic-Veljan2008}, it remains to check the first $13,069$ values of $\mathscr{R}^2 \mathscr{L}(B_n)$. We have verified that these initial values are all greater than one with {\tt Mathemtaica}. It follows that $\{B_n\}_{n\geq 1}$ is $2$-log-convex.
\end{proof}

With the aid of the package {\tt P-rec.m}, it is easy to get that $\{B_n\}_{n\geq 882712}$ is $3$-log-convex. We do not prove the $3$-log-convexity of $\{B_n\}_{n\geq 1}$ due to limitations of our equipment.

\section{Conclusion}\label{sec-conclusion}

To conclude this paper, we present three conjectures concerning the $(q,t)$-Hoggatt sums $H_n^{[m]}(q,t)$ and the Baxter numbers $B_n$.

In his PhD thesis, Dilks~\cite{Dilks2015} studied the specialization of $H_n^{[m]}(q,t)$ to $m=1,2,3$ and presented their explicit expressions.
For $m=1$, the $(q,t)$-Hoggatt sum is
$$H_n^{[1]}(q,t)= \sum_{k=0}^{n-1}q^{\binom{k+1}{2}}\qbinomial{n-1}{k}t^k=(1+tq)(1+tq^2)\cdots(1+tq^{n-1}),$$
and it is clear that $\{H_n^{[1]}(q,t) \}_{n\ge 1}$ is a generalized Sturm sequence for any fixed positive integer~$q$.
For $m\ge 2$, Theorem~\ref{thm:interlace-HP} states that when $q=1$, the polynomials $H_n^{[m]}(q,t)$ form a Sturm sequence.
In fact, we observe that this property extends to any positive integer $q$ for $m=2,3$. This leads to the following conjecture which we have verified for $1\le q \le 10$ and $1\le n\le 15$ by {\tt Mathematica}.
\begin{conj}\label{conj-1}
	For any integer $q>0$ and $m=2,3$, the polynomial sequence $\{H_n^{[m]}(q,t)\}_{n\ge 1}$ forms a Sturm sequence.
\end{conj}

For Baxter numbers, Theorem~\ref{thm:asym-r-lgcvBn} suggests the following conjecture.
\begin{conj}\label{conj-2}
	The sequence of Baxter numbers $\{B_n\}_{n\geq 1}$ is infinitely log-convex.
\end{conj}

It is known that a Stieltjes moment sequence is infinitely log-convex~\cite{WangZhu2016}. However, we observe that the sequence $\{B_n\}_{n\ge 1}$ fails to be a Stieltjes moment sequence, since the fifth-order leading principal minor of its Hankel matrix is negative.

Note that the Baxter numbers can be treated as the specialization of the classical Hoggatt sums $H_n^{[m]}(1,1)$ at $m=3$. Therefore, we propose a more general conjecture as follows.

\begin{conj}\label{conj-H-vex}
	The sequence of $(1,1)$-Hoggatt sums $\{H_n^{[m]}(1,1)\}_{n\ge 1}$ is infinitely log-convex for any fixed $m\geq 2$.
\end{conj}

The conjecture holds for $m=2$ since the numbers $H_n^{[2]}(1,1)$ are the Catalan numbers. This sequence is known to be a Stieltjes moment sequence and hence is infinitely log-convex~\cite{Liang2016}. For any fixed $m\geq4$, one may adapt a similar method used in the proof of Theorem~\ref{thm:asym-r-lgcvBn} to obtain the asymptotic $r$-log-convexity of $H_n^{[m]}(1,1)$. However, we have not found a unified approach that works for arbitrary $m$. Applying the package {\tt P-rec.m}, one can obtain that $\{H_n^{[4]}(1,1)\}_{n\ge 9632}$ is log-convex and $\{H_n^{[4]}(1,1)\}_{n\ge 404602}$ is $2$-log-convex. We have verified that all initial $9,632$ values of $\mathscr{R}^2(H_n^{[4]}(1,1))$ exceed one. Thus, $\{H_n^{[4]}(1,1)\}_{n\ge 1}$ is log-convex. We do not check the initial $404,602$ values of $\mathscr{R}^2\mathscr{L}(H_n^{[4]}(1,1))$ because the limitations of our computers. All of the above results suggest that Conjecture~\ref{conj-H-vex} is probably valid.

\section*{Acknowledgments}
We would like to thank Shaoshi Chen, Qing-Hu Hou, Arthur L.B. Yang and Zuo-Ru Zhang for valuable suggestions and helpful comments.
Hanqian Fang and Candice X.T. Zhang were partially supported by the National Key R\&D Program of China (No. 2023YFA1009401), the NSFC grant (No. 12271511), the CAS Funds of the Youth Innovation Promotion Association (No. 2022001), and the Strategic Priority Research Program of the Chinese Academy of Sciences (No. XDB0510201). Candice X.T. Zhang was also supported by the Postdoctoral Fellowship Program of CPSF under Grant Number: GZC20241868.

\end{document}